\documentclass[11pt,a4paper]{article}
\usepackage[top=1in, bottom=1.25in, left=1in, right=1in]{geometry}
\usepackage[utf8]{inputenc}
\usepackage[T1]{fontenc}
\usepackage[english]{babel}
\usepackage{graphicx,epstopdf}
\usepackage{caption, subcaption,float}
\usepackage[dvipsnames]{xcolor}
\usepackage{multirow}
\usepackage{tikz}
\usepackage{amsthm}
\usepackage{tabls}
\usepackage{float}
\usepackage[round,sort,comma]{natbib}
\usepackage{setspace}
\usepackage[ruled, linesnumberedhidden,vlined]{algorithm2e} 
\usetikzlibrary{matrix}
\usepackage{amsmath,amssymb,dsfont,mathtools}
\usepackage{quiver}

\usepackage{hyperref}
\hypersetup{
    colorlinks=true,       
    linkcolor=blue,          
    citecolor=olive,       
    filecolor=magenta,      
    urlcolor=blue          
}
\usepackage{aliascnt}

\begin{document}

\title{Parabolic subgroups of large-type Artin groups}
\date{}
\author{Mar\'{i}a Cumplido, Alexandre Martin, and Nicolas Vaskou}


\maketitle
\theoremstyle{plain}
\newtheorem{theorem}{Theorem}

\newtheorem{theoremalph}{Theorem}
\renewcommand*{\thetheoremalph}{\Alph{theoremalph}}

\newtheorem{corollaryalph}{Corollary}
\renewcommand*{\thecorollaryalph}{\Alph{corollaryalph}}

\newaliascnt{lemma}{theorem}
\newtheorem{lemma}[lemma]{Lemma}
\aliascntresetthe{lemma}
\providecommand*{\lemmaautorefname}{Lemma}

\newaliascnt{proposition}{theorem}
\newtheorem{proposition}[proposition]{Proposition}
\aliascntresetthe{proposition}
\providecommand*{\propositionautorefname}{Proposition}

\newaliascnt{corollary}{theorem}
\newtheorem{corollary}[corollary]{Corollary}
\aliascntresetthe{corollary}
\providecommand*{\corollaryautorefname}{Corollary}

\newaliascnt{conjecture}{theorem}
\newtheorem{conjecture}[conjecture]{Conjecture}
\aliascntresetthe{conjecture}
\providecommand*{\conjectureautorefname}{Conjecture}

\newtheorem*{question*}{Question}

\theoremstyle{remark}

\newaliascnt{claim}{theorem}
\newtheorem{claim}[claim]{Claim}

\newaliascnt{notation}{theorem}
\newtheorem{notation}[notation]{Notation}
\aliascntresetthe{notation}
\providecommand*{\notationautorefname}{Notation}

\aliascntresetthe{claim}
\providecommand*{\claimautorefname}{Claim}

\newtheorem*{claim*}{Claim}
\theoremstyle{definition}

\newaliascnt{definition}{theorem}
\newtheorem{definition}[definition]{Definition}
\aliascntresetthe{definition}
\providecommand*{\definitionautorefname}{Definition}

\newaliascnt{remark}{theorem}
\newtheorem{remark}[remark]{Remark}
\aliascntresetthe{remark}
\providecommand*{\exampleautorefname}{Remark}

\newaliascnt{example}{theorem}
\newtheorem{example}[example]{Example}
\aliascntresetthe{example}
\providecommand*{\exampleautorefname}{Example}


\def\autorefspace{\hspace*{-0.5pt}}
\def\sectionautorefname{Section\autorefspace}
\def\subsectionautorefname{Section\autorefspace}
\def\subsubsectionautorefname{Section\autorefspace}
\def\figureautorefname{Figure\autorefspace}
\def\subfigureautorefname{Figure\autorefspace}
\def\tableautorefname{Table\autorefspace}
\def\equationautorefname{Equation\autorefspace}
\def\Itemautorefname{item\autorefspace}
\def\Hfootnoteautorefname{footnote\autorefspace}
\def\AMSautorefname{Equation\autorefspace}

\newcommand{\quotient}[2]{{\raisebox{.2em}{$#1$}\left/\raisebox{-.2em}{$#2$}\right.}}
\def\Stab{{\rm Stab}}
\def\Fix{{\rm Fix}}

\newcommand{\heyMaria}[1]{{\color{RedOrange} (#1)}}
\newcommand{\heyNicolas}[1]{{\color{Purple} (#1)}}
\newcommand{\heyAlex}[1]{{\color{OliveGreen} (#1)}}

\newcommand{\myref}[2]{\hyperref[#1]{#2~\ref*{#1}}}

\begin{abstract}
We show that the geometric realisation of the poset of proper parabolic subgroups of a large-type Artin group has a systolic geometry. We use this geometry  to show that the set of parabolic subgroups of a large-type Artin group is stable under arbitrary intersections and forms a lattice for the inclusion. As an application, we show that parabolic subgroups of large-type Artin groups are stable under taking roots and we completely characterise the parabolic subgroups that are conjugacy stable. 

We also use this geometric perspective to recover and unify results describing the normalisers of parabolic subgroups of large-type Artin groups. 
%

\medskip

{\footnotesize

\noindent \emph{Key words.}} Artin groups; parabolic subgroups; systolic complexes.

\end{abstract}

\section{Introduction}
Artin groups are a class of groups strongly related to Coxeter groups, and defined as follows: Let $S$ be a finite set, and for every distinct $s, t \in S$,  choose an integer $m_{st} \in \{2, 3, \ldots,\infty \}$. The associated \textbf{Artin group} is given by the following presentation: 
$$A_S \coloneqq \langle S ~|~ \underbrace{sts\cdots}_{m_{st}} = \underbrace{tst\cdots}_{m_{st}}  \mbox{ when } m_{st} \neq \infty\rangle.$$
If we add the relations $s^2 = 1$ for all $s \in S$, we obtain the associated \textbf{Coxeter group} $W_S$. Every Artin group $A_S$ has an associated \textbf{Coxeter graph}~$\Gamma_S$ defined as follows:

\begin{itemize}

\item The set of vertices of $\Gamma_S$ is $S$.

\item There is an edge connecting $s$ and $t$ if and only if $m_{s,t}\neq \infty$. This edge is labelled with~$m_{s,t}$.

\end{itemize}


\medskip

Many  questions remain open for general Artin groups, such as whether they are torsion-free, whether they have a soluble word problem, or whether the satisfy the celebrated $K(\pi,1)$-conjecture. However, several classes of Artin groups are better understood, for instance: \textbf{right-angled} Artin groups ($m_{st}=2$ or $\infty$ for all $s,t \in S$),   Artin groups \textbf{of spherical type} (such that the associated Coxeter group $W_S$ is finite)  and Artin groups of \textbf{large type} ($m_{ab} \geq 3$ for all $a, b \in S$).

\medskip

The goal of this paper is to investigate the structure of certain subgroups of large-type Artin groups. For a general Artin group $A_S$ with generating set $S$, it is a theorem of  \cite{Vanderlek} that the subgroup generated by a subset $S'\subset S$ is isomorphic to the Artin group~$A_{S'}$. The various subgroups $A_{S'}$, for subsets~$S'$ of~$S$, are called the \textbf{standard parabolic subgroups} of~$A_S$, and their conjugates are the \textbf{parabolic subgroups} of~$A_S$. A parabolic subgroup conjugated to a standard parabolic subgroup~$A_{S'}$ will be said to be of \textbf{type~$\boldsymbol{S'}$}. An Artin group that does not decompose as the direct product of two of its standard parabolic subgroups is called \textbf{irreducible}. Since a parabolic subgroup can naturally be viewed as an Artin group by the above, one defines similarly the notion of \textbf{irreducible} parabolic subgroup.

Parabolic subgroups form a natural class of subgroups that has been playing an increasing role in the geometric study of Artin groups in recent years. Indeed, several complexes have been associated to Artin groups using the combinatorics of parabolic subgroups. For instance, Deligne complexes and their variants are built out of (cosets of) standard parabolic subgroups of spherical type \citep{CharneyDavis}, and have been used to study various aspects of Artin groups: $K(\pi, 1)$-conjecture \citep{CharneyDavis,ParisSurvey}, acylindrical hyperbolicity \citep{MP1, Vaskou1, CMW}, Tits alternative \citep{MP2}, etc. More recently, using the connections between braid groups and mapping class groups, the irreducible parabolic subgroups have been used to define a possible analogue of the complex of curves for Artin groups of spherical type \citep{CGGW,Rose}. The geometry of this complex is currently being intensively studied. 

\medskip

The combinatorics of the set of parabolic subgroups of Coxeter groups are well understood. For instance, it is known that  the intersection of any subset of parabolic subgroups  of a Coxeter group is itself a parabolic subgroup \citep{Qi}. This implies in particular that the set of parabolic subgroups is a lattice for the inclusion. By contrast, the analogous problem is open  for general Artin groups: 

\begin{question*}
	Let $A_S$ be a general Artin group. Is the set of parabolic subgroups stable under arbitrary intersections?
\end{question*}

The answer to this question is known for braid groups: A braid group can be seen as the mapping class group of an $n$-punctured disc $\mathcal{D}_n$. In this situation, parabolic subgroups are in bijection with isotopy classes of non-degenerated simple closed multicurves, each of them defining a disjoint union of (at least $2$-punctured) discs in  $\mathcal{D}_n$. An intersection between these families of discs can be defined (see {\citet[Section~1]{FM}} to get an idea of the construction). This corresponds to the intersection of parabolic subgroups of the braid group and gives us an affirmative answer to our question. This answer was recently generalised to all Artin groups of spherical type by \cite{CGGW} using Garside theory. For so-called Artin groups \textbf{of type FC}, it was shown that the intersection of two parabolic subgroups \textbf{of spherical type} is again a parabolic subgroup of spherical type \citep{Rose}. However, the case of general parabolic subgroups remains open.

\medskip

Besides being interesting in their own right, such results about the poset of parabolic subgroups  can be valuable tools in studying the structure of Artin groups: For instance, the analogue of \myref{Theorem:A}{Theorem} for Artin groups of spherical type was a key ingredient in the proof that Artin groups of type FC satisfy the Tits alternative \citep{MP2} . 

\bigskip

 In this paper, we solve this problem for Artin groups of large-type:


\begin{theoremalph}\label{Theorem:A}
	Let $A_S$ be a large-type Artin group. Then the intersection of an arbitrary  subset of parabolic subgroups of an $A_S$ is itself a parabolic subgroup.  
	Moreover, the set of parabolic subgroups of $A_S$ is a lattice for the inclusion.
\end{theoremalph}

Note that a consequence of this theorem is that every subset of $A_S$ is contained in a unique minimal parabolic subgroup: This generalises to large-type Artin groups the notion of \textbf{parabolic closure} known for Coxeter groups \citep{Qi} and Artin groups of spherical type \citep{CGGW}. 

%
%

\bigskip

The approach in this article is geometric in nature. We associate to each Artin group $A_S$ a simplicial complex~$X_S$, called its \textbf{Artin complex}, whose first barycentric subdivision is exactly the geometric realisation of the poset of proper parabolic subgroups of $A_S$. In essence, the Artin complex $X_S$ is the complex obtained by modifying the construction of the Deligne complex in order to allow \textit{all} proper standard parabolic subgroups instead of those of spherical type  (see \autoref{section2} for more details). The advantage in considering this complex is that all the parabolic subgroups of $A_S$ arise as stabilisers of simplices of $X_S$ and can thus be studied geometrically. In particular, studying intersections of parabolic subgroups can be done if we have a sufficiently strong control over the (combinatorial) geodesics of $X_S$ between two simplices. This is possible for large-type Artin groups, as we show that these complexes are non-positively curved in an appropriate sense. The key geometric result of this article is the following:

\begin{theoremalph}\label{thm:main_systolic}
	Let $A_S$ be a large-type Artin group on at least $3$ generators. Then its Artin complex~$X_S$  is systolic.	
\end{theoremalph}

Large-type Artin groups were recently shown by \cite{HO} to be systolic groups. However, we emphasise that the systolic geometry appearing here is of a rather different nature: The systolic complex associated to $A_S$ considered by Huang-Osajda is essentially a (thickened)  Cayley graph of~$A_S$ for the standard generating set, and as such is quasi-isometric to~$A_S$. By contrast, the Artin complex~$X_S$ studied here is quasi-isometric to the Cayley graph of~$A_S$ with respect to {all its  proper parabolic subgroups, and in particular the action of~$A_S$ on~$X_S$ is cocompact but far from being proper.

\bigskip

As an application, we solve the conjugacy stability problem for parabolic subgroups of large-type Artin groups.  A subgroup $H$ of a group $G$ is \textbf{conjugacy stable} if for every pair of elements $a,b\in H$ such that $a = \alpha^{-1} b \alpha $ there is $\beta\in H$ such that $a = \beta^{-1} b \beta $.  A natural question to ask is which parabolic subgroups of an Artin group are conjugacy stable. 
This problem had already been solved for parabolic subgroups of spherical Artin groups \citep{CCC}, generalizing pre-existing results for braids of \cite{Meneses}.
We answer this question for large-type Artin groups:

\begin{theoremalph}\label{conjugacystability}
	Let $A_X$ be a standard parabolic subgroup of a large-type Artin group $A_S$.
	 Then~$A_X$ is not conjugacy stable in~$A_S$ if and only if there exist vertices~$a$ and~$b$ of $\Gamma_X$ that are connected by an odd-labeled path in~$\Gamma_S$ and that are  not connected by an odd-labeled path in~$\Gamma_X$. 
\end{theoremalph}

\noindent
Notice that conjugacy stability is preserved under subgroup conjugation, hence the previous theorem classifies all parabolic subgroups of a large-type Artin group under conjugacy stability.	

\bigskip

As another application, we show that parabolic subgroups of large-type Artin groups are stable under taking roots, whose analogue for Artin groups of spherical type was proved in \citep[Corollary~8.3]{CGGW}

\begin{theoremalph}\label{roots}
	Let $A_S$ be a large-type Artin group, let $P$ be a  parabolic subgroup of $A_S$, and let $g \in A_S$. If $g^n \in P$ for some non-zero integer $n$, then $g \in P$.	
	\end{theoremalph}
	
Beside the intersection properties of parabolic subgroups, the previous result relies on understanding the fixed-point sets and normalisers of parabolic subgroups. Their structure has been studied by various authors as we explain below, but the results are a bit hidden in the literature. In the case of large-type Artin groups, our approach provides a unifying perspective that allows us to recover all these results within a single framework. We mention the full result here for ease of reference and we re-prove it with our techniques, as we believe such results on normalisers of parabolic subgroups are of independent interest:

\begin{theoremalph}\label{thm:main_normalizer}
		Let $A_S$ be a large-type Artin group and let $P$ be a parabolic subgroup of type $S'$.
	\begin{itemize}
		\item If $|S'|\geq 2$, then $N(P)=P.$
		\item If $|S'| = 1$, then $N(P)$ splits as a direct product of the form
		$$N(P)= P \times F,$$
		where $F$ is a finitely-generated  free group. Moreover, there is an explicit description of a basis of $F$ (see  \autoref{prop:basis_free} for details).
	\end{itemize}
	\end{theoremalph}
	
The structure of normalisers of parabolic subgroups in Artin groups of large type had already been investigated by Luis Paris and Eddy Godelle, although it is a bit hidden in their papers. Recall that an Artin group that cannot be decomposed as the direct product of two of its standard parabolic subgroups is called irreducible. In \citep[Section~4]{Paris}, the conjugation of standard parabolic subgroups is described by an algorithm. In particular, we know that the only pairs of different irreducible standard parabolic subgroups that can be conjugated are the spherical ones.  In the large case, as all parabolic subgroups are irreducible and the only spherical parabolic subgroups are the \textbf{dihedral} ones (i.e. the parabolic subgroups on  two generators), the situation is as follows: $A_X$ and $A_{X'}$ are conjugate if and only if $X=X'$ or $X=\{a\}$, $X'=\{b\}$ and~$a$ and~$b$ are connected in~$\Gamma_S$  by an odd-labelled path.  \cite[Definition~4.1, Corollary~4.12]{Godelle2} tell us that the conjugating elements between two (possibly equal) standard parabolic subgroups~$A_X$ and~$A_{X'}$ must be the product of an element in~$A_X$ and an element associated to the previous path. If $|X|>1$, such a path does not exists and then $N(A_X)=A_X$. If $|X|=1$, the description of the normaliser is similar to the one given in \autoref{prop:basis_free}. However, the description Godelle gives there is set-theoretic and does not describe the  direct product structure.

 The structure of the normaliser of cyclic parabolic subgroups for large-type Artin groups (and more generally two-dimensional Artin groups) had been obtained, albeit under a different name, in \citep[Proposition 4.5]{MP1}. Moreover, a basis of the corresponding free group had been stated as a remark, but without details. 
 
 \bigskip 

The paper is organised as follows. In \autoref{section2}, we introduce the Artin complex of a general Artin group, and show that its local structure is particularly well-behaved: The links of simplices are themselves (smaller) Artin complexes, see \autoref{LinkDev}. In \autoref{section3}, we use this local structure to prove  \myref{thm:main_systolic}{Theorem}. \autoref{section4} exploits the systolic geometry of the Artin complex to prove  \autoref{theorem:intersection}. In \autoref{section5}, we study the geometry of fixed-point sets of parabolic subgroups in order to prove \myref{thm:main_normalizer}{Theorem}. Finally, we prove \myref{conjugacystability}{Theorem} and \myref{roots}{Theorem} in \autoref{section6}.

\section{The Artin complex}\label{section2}

The goal of this section is to introduce our main geometric object: the Artin complex associated to an Artin group. Later on, we present some of its basic properties. When talking about complexes of groups, we will use the notations of \citep[Chapter~II.12]{BH}.

\begin{definition} \label{DefArtComp} Consider an Artin group $A_S$ with $|S| \geq 2$, and a simplex $K$ of dimension $|S| - 1$. We define a simplex of groups over $K$ as follows. The simplex $K$ is given a trivial local group. There is a one-to-one correspondance between the elements $s_i \in S$ and the codimension $1$ faces of $K$, and we denote by $\Delta_{ s_i }$  these codimension $1$ faces. In particular, $\Delta_{ s_i }$ is given the local group~$\langle s_i \rangle$. Changing the codimension, there is a bijection between the strict subsets of $S$ and the faces of $K$. Every face of $K$ of codimension $k$ can be written uniquely as the intersection
$$\Delta_{S'} := \bigcap\limits_{s_i \in S'} \Delta_{ s_i } \text{ for some } S' \subsetneq S \text{ with } |S'| = k.$$
The face $\Delta_{S'}$ is then given the local group $A_{S'}$.

\smallskip

The morphism associated to an inclusion of faces $K_{S''} \subset K_{S'}$ is the natural inclusion $\psi_{S'S''} : A_{S''} \hookrightarrow A_{S'}$. Let $\mathcal{P}$ be the poset of standard parabolic subgroups of~$A_S$ ordered with natural inclusion. As each~$A_{S'}$ is itself an Artin group \citep{Vanderlek}, there is a simple morphism, $\varphi: G(\mathcal{P}) \hookrightarrow A_S$, given by inclusion, from the complex of groups to the Artin group. The complex $X_S \coloneqq D_{K}(\mathcal{P},\varphi)$ obtained by development of $\mathcal{P}$ over $K$ along $\varphi$ is called the \textbf{Artin complex} associated to $A_S$ (see \citep[Theorem~II.12.18]{BH} for the definition of development, see also the remark below).

Note that the action of $A_S$ on $X_S$ is without inversions and cocompact, with strict fundamental domain a single simplex which is isomorphic to $K$. To avoid any confusion, we will from now on denote by $\overline{K}$ the quotient space and by $\overline{\Delta}_{S'}$ its faces, and we will denote by $K$ a chosen fundamental domain of $X_S$ and by $\Delta_{S'}$ its faces. For every simplex $\Delta$ of $X_S$, there is a unique subset $S' \subsetneq S$ such that $\Delta$ is the same orbit as $\Delta_{S'}$. We say that such a simplex is \textbf{of type~$\boldsymbol{S'}$}. 
\end{definition}

\begin{remark}
In \cite[Proof of Theorem~II.12.18]{BH}, the authors give a topological description of the spaces obtained by development of such complexes of groups. In light of this, the Artin complex $X_S$ can also be described by the following:
$$X_S \coloneqq D_{K}(\mathcal{P},\varphi) \coloneqq \quotient{A_S \times K}{\sim},$$
where $(g,x) \sim (g',x') \Longleftrightarrow x = x'$ and $g^{-1} g'$ belongs to the local group of the smallest simplex of $K$ containing $x$.
\end{remark}

\begin{remark}
 Another perhaps more intuitive  way to look at $X_S$ is the following. Consider the poset of proper parabolic subgroups of $A_S$ and its geometric realisation $P_S$, defined as follows:
	\begin{itemize}
		\item The vertex set of $P_S$ is the set of \textit{proper} parabolic subgroups of $A_S$.
		
		\item There is a $(n-1)$-simplex between vertices of $P_S$ corresponding to proper parabolic subgroups $P_1, \ldots, P_n$ whenever there is a sequence of inclusions $P_n \subset \cdots \subset P_1$. This happens if and only if there is an element $g \in A_S$ and proper subsets $S^{(n)} \subsetneq \cdots \subsetneq S'$ of $S$ such that $P_i = gA_{S^{(i)}}$. 
	\end{itemize}
	
	\noindent
	Then $P_S$ is exactly the barycentric subdivision of $X_S$.
\end{remark}

\begin{lemma} \label{pi0pi1} Let $A_S$ be an Artin group and let $X_S$ be its Artin complex. Then $X_S$ is connected. Additionally, if $|S| \geq 3$, then $X_S$ is simply-connected.
\end{lemma}

\begin{proof} This is a direct consequence of \citep[Chapter~II.12, Proposition~12.20]{BH}. $X_S$ is connected because the Artin group $A_S$ is generated by its standard parabolic subgroups. Moreover, if $|S| \geq 3$, then $A_S$ is the colimit of its standard parabolic subgroups, by \citep{Vanderlek}, and thus $X_S= D_\Delta(\mathcal{P}, \varphi)$ is the  universal cover of the complex of groups $G(\mathcal{P})$, hence is simply-connected.
\end{proof}

\begin{definition} \label{DefLink} Let $Y$ be a simplex in a simplicial complex $X$. The \textbf{link} of $Y$ in $X$ is the simpicial complex~$Lk_X(Y)$ consisting of the simplices of $X$ that are disjoint from $Y$ and which together with $Y$ span a simplex of~$X$.
\end{definition}

\begin{lemma} \label{LinkDev} Let $A_S$ be an Artin group with Artin complex $X_S$. Then the link of a simplex of type $S'$ is isomorphic to the Artin complex~$X_{S'}$ associated to the Artin group~$A_{S'}$.
\end{lemma}

\begin{proof} By \citep[Chapter~II.12, Construction 12.24]{BH}, it is possible to describe the link of a simplex in the development of a complex of groups as the development of an appropriate subcomplex of groups. We explain below how this applies to $X_S$.

The link of $\overline{\Delta}_{S'}$ in $\overline{K}$ is a simplex of dimension $|S'|-1$, whose poset of faces is isomorphic to the poset  of proper subsets of $S'$ ordered with the inclusion. The complex of groups $G(\overline{K})$ induces a complex of groups on the link $Lk_{\overline{K}}\big(\overline{\Delta}_{S'}\big)$. Moreover, there is a simple morphism $\varphi_{S'} : G(Lk_{\overline{K}}\big(\overline{\Delta}_{S'}\big)) \rightarrow A_{S'}$ given by the family of homomorphisms $$(\varphi_{S'})_{S''}: A_{S''} \underset{\psi_{S'S''}}{\longrightarrow} A_{S'}.$$
It follows from the construction described in \citep[Chapter~II.12, Construction 12.24]{BH} that the link of $Lk_{X_S}(\Delta_{S'})$  is isomorphic to the development $D(Lk_{\overline{K}}\big(\overline{\Delta}_{S'}\big),\varphi_{S'})$. Note that the induced complex of groups on $Lk_{\overline{K}}\big(\overline{\Delta}_{S'}\big)$ is naturally isomorphic to the complex of groups associated to $A_{S'}$ in \autoref{DefArtComp}. Moreover, the simple morphism $\varphi_{S'}$ coincides with the simple morphism used in \autoref{DefArtComp} to define the Artin complex $X_{S'}$. Putting everything together, it now follows that the link $Lk_{X_S}(\Delta_{S'})$ is isomorphic to $X_{S'}$.

This argument generalises in a straightforward way to any simplex $g\Delta_{S'}$ of $X_S$ of type $S'$.
%
%
%
%
\end{proof}

\section{Systolicity}\label{section3}


The goal of this section is to prove  \myref{thm:main_systolic}{Theorem}. Recall that a combinatorial path $\gamma$ in a simplicial complex $X$ is \textbf{full} if two vertices of $\gamma$ adjacent in $X$ are always adjacent in $\gamma$. If $\gamma$ is in the $1$-skeleton of $X$, then the \textbf{simplicial length} of $\gamma$ is the number $|\gamma|$ of edges contained in $\gamma$. We will denote by $\Stab(T)$ or $\Stab_{X_S}(T)$ the stabiliser of a set of points $T$ in $X_S$. We introduce a few more definitions from systolic geometry \citep{JS}:

\begin{definition} The \textbf{systole} of a simplicial complex $X$ is
$$\operatorname{sys}(X) \coloneqq \min \{ |\gamma| \ | \ \gamma \text{ is an embedded full cycle of } X \} \in \{3,4,\cdots, \infty\}.$$
For $k \in \{3, \ldots, \infty\}$, we say that $X$ is \textbf{locally k-large} if $\operatorname{sys}(Lk_X(Y)) \geq k$ for all simplices $Y \subseteq X$. We say that $X$ is \textbf{k-large} if it is locally $k$-large and $\operatorname{sys}(X) \geq k$.  $X$ is \textbf{k-systolic} if it is connected, simply-connected and locally $k$-large. Finally, $X$ is called \textbf{systolic} if it is $6$-systolic.
\end{definition}

\noindent The main result of this section is the following:

\begin{theorem} \label{MainThm} Let $A_S$ be an Artin group with $|S| \geq 3$. If all coefficients in $A_S$ are at least $k \in \{3, \ldots, \infty\}$, then its Artin complex $X_S$ is $2k$-systolic. In particular, if $A_S$ is of large type, then $X_S$ is systolic.
\end{theorem}

\noindent In order to prove this theorem, we need the following lemma:

\begin{lemma} \label{MainLemma} Let $A_S$ be an Artin group on two generators $a, b$ with coefficient $m_{ab}\in \{3, \ldots, \infty\}$ and Artin complex $X_S$. Then $\operatorname{sys}(X_S) = 2m_{ab}$.
\end{lemma}

\begin{proof} If $m_{ab}=\infty$, it follows directly from the definition of the Artin complex that $X_S$ is the Bass-Serre tree associated to the splitting $\langle a \rangle * \langle b \rangle$. The result is then immediate.
	
	Let us now assume that $m_{ab} < \infty$. Let $e$ be the edge in $X$ whose vertices $x$, $y$ correspond to the cosets $\langle a \rangle$ and $\langle b \rangle$. Let $\gamma$ be a non-backtracking loop in $X_S$. Since $X_S$ is a bipartite graph coloured by the cosets of $\langle a \rangle$ and $\langle b \rangle$ respectively, the length of $\gamma$ is even. Denote by $e_0 , e_1, \ldots, e_k$ the edges of $\gamma$. Since the action of $A_S$ on $X_S$ is transitive on edges, let us assume that $e_0 = e$.
	
	Note that the action of $\langle a \rangle$  is transitive on the set of edges around $x$, and so is the action of $\langle b \rangle$ on the edges around $y$. Assume without loss of generality that $\gamma$ first goes through $x$, i.e. $e_1$ and $e_0$ share the vertex $x$.
	Then $e_1$ must be of the form $a^{r_1} e$, for some $r_1 \in \mathbf{Z} \backslash \{0\}$. Note that the edges $e_1$ and $e_2$ then share the vertex~$a^{r_1}y$. 
	The action of $a^{r_1} \langle b \rangle a^{-r_1}$ is transitive on the set of edges around~$a^{r_1}y$, thus $e_2$ must of the form $a^{r_1} b^{r_2} e$, for some $r_2 \in \mathbf{Z} \backslash \{0\}$.
	We continue this process by induction until $\gamma$ stops. 
In particular, the final edge $e_k$ is of the form 
$$a^{r_1} b^{r_2} \cdots a^{r_{k-1}} b^{r_k}$$
for non-zero integers $r_1, \ldots, r_k$. 
 But since $e_k = e$ as $\gamma$ is a loop, we get $a^{r_1} b^{r_2} \cdots a^{r_{k-1}} b^{r_k} e = e$. Since $\operatorname{Stab}(e) = \{1\}$, it follows that $a^{r_1} b^{r_2} \cdots a^{r_{k-1}} b^{r_k}$ must be trivial in $A_S$. But it is also a non-trivial word, as $\gamma$ is not homotopically trivial. By \citep[Lemma~6]{AS}, we must have $k \geq 2m_{ab}$. Hence, the combinatorial length of $\gamma$ is $|\gamma| = k \geq 2 m_{ab}$.
\end{proof}

\noindent We can now prove the main theorem:

\begin{proof}[Proof of \autoref{MainThm}] We will prove by induction on the number $|S|$ of generators of the Artin groups $A_S$ that their associated Artin complexes $X_S$ are $2k$-systolic.

If $|S| = 3$, we know from \autoref{pi0pi1} that $X_S$ is connected and simply connected. It only remains to show that for all $g \in A_S$, for all $S' \subsetneq S$, the simplex $g \cdot \Delta_{S'}$ is such that $Lk_{X_S}(g \cdot \Delta_{S'})$ is $2k$-large. If $|S'| = 2$, then the link $Lk_{X_S}(g \cdot \Delta_{S'})$ is isomorphic to the Artin complex $X_{S'}$ associated to the Artin group $A_{S'}$ (\autoref{LinkDev}), and the latter is $2k$-large by \autoref{MainLemma}. The cases $|S'| = 0$ or $1$ are trivial.

Let us now assume that $|S| > 3$ and that every Artin complex $A_{S'}$ with $S' \subsetneq S$ is $2k$-systolic. Again, we know from \autoref{pi0pi1} that $X_S$ is connected and simply connected, so it only remains to show that for all $g \in A_S$, for all $S' \subsetneq S$, the simplex $g \cdot \Delta_{S'}$ is such that $Lk_{X_S}(g \cdot \Delta_{S'})$ is $2k$-large. If $|S'| \geq 2$, then $Lk(g \cdot \Delta_{S'},X_S)$ is isomorphic to the Artin complex $X_{S'}$ associated to the Artin group $A_{S'}$ (\autoref{LinkDev}). The latter is $2k$-systolic by the induction hypothesis, hence is $2k$-large as well \citep[Proposition~1.4]{JS}. Once again, the cases $|S'| = 0$ or $1$ are trivial.
\end{proof}
\section{Intersection of parabolic subgroups}\label{section4}

The aim of this section is to use the systolicity of the Artin complex of an Artin group of large type to prove  \myref{Theorem:A}{Theorem}. We will do it by proving the following theorem:


\begin{definition} Let $P_1$ and $P_2$ be two parabolic subgroups of an Artin group $A_S$ such that $P_1\subseteq P_2$. We say that $P_1$ is a parabolic subgroup of $P_2$ if $P_1\subseteq P_2$ is conjugate to an inclusion of standard parabolic subgroups $A_{S''}\subseteq A_{S'}$, $S''\subseteq S'$.
\end{definition}

\begin{theorem}\label{theorem:intersection} Let $A_S$ be an Artin group of large-type.

\begin{enumerate}

\item\label{item1} The intersection of two parabolic subgroups of $A_S$ is again a parabolic subgroup of $A_S$.

\item\label{item2} If $P_1$ and $P_2$ are two parabolic subgroups of $A_S$ such that $P_1\subseteq P_2$, then $P_1$ is a parabolic subgroup of $P_2$.

\end{enumerate} 
\end{theorem}


\noindent
Note that the second item in the previous theorem is already a result of  \citet[Theorem~3]{Godelle2}. However, we believe the reader may be interested in recovering this result directly from our perspective.

\bigskip

In all this section, $A_S$ denotes an Artin group on at least $3$ generators. First notice the Artin complex allows us to understand geometrically the parabolic subgroups of $A_S$, via the following correspondence:

\begin{lemma}\label{lemma:geometric_parabolics}
Let $A_S$ be an Artin group on at least 3 generators and let~$X_S$ be its associated Artin complex. Then
\begin{itemize}

\item The parabolic subgroups of $A_S$ are exactly the stabilisers of simplices of $X_S$.

\item Let $\Delta$ be a simplex of $X_S$. The parabolic subgroups of~$\Stab_{X_S}(\Delta)$ are exactly the stabilisers of the simplices that contain $\Delta$.

\end{itemize}

\end{lemma} 

\begin{proof} 
By construction, every standard parabolic subgroup $A_{S'}$ is precisely the stabiliser of some simplex $\Delta_{S'}$ lying on the fundamental domain~$K$ of $X_S$, and viceversa. Moreover, any parabolic subgroup of the form $g A_{S'} g^{-1}$ is the stabiliser of the simplex $g\cdot \Delta_{S'}$, $g\in A$. To prove the first claim, notice that any simplex of $X_S$ can be expressed as $g' \cdot \Delta'$, where $\Delta'$ is in~$K$ and $g'\in A$.  

\medskip

Let us now prove the second claim. On the one hand, let $P$ be a parabolic subgroup of $\Stab_{X_S}(\Delta)$. Up to conjugation, we can suppose that $\Delta$ lies in~$K$ of $X_S$, and that~$P$ is the stabiliser of a simplex $\Delta'$ that also lies in $K$. Now notice that, by construction of the fundamental domain, this implies that $\Delta'$ contains $\Delta$, as we desired. On the other hand, note that if~$\Delta''$ is a simplex that contains~$\Delta$, then we can find an element $g\in A_S$ such that $g\cdot \Delta''$ belongs to~$K$. Hence $g'\Stab_{X_S}(\Delta'')g'^{-1}\subseteq g'\Stab_{X_S}(\Delta)g'^{-1}$ is an inclusion of standard parabolic subgroups, as we wanted to prove.
\end{proof}

\begin{remark}
The previous correspondence is not a bijection between the parabolic subgroups of~$A_S$ and the simplices of its Artin complex, as two distinct simplices may have the same stabiliser.
\end{remark}

\medskip\noindent
Secondly, we mention the following result from systolic geometry, well-known to experts, that will be used in our proof:

\begin{lemma}\label{lemma:fixing_geodesic}
Let $G$ be a group acting without inversions on a systolic complex $Y$, and let $H$ be a subgroup of $G$. Suppose that $H$ fixes two vertices $v$ and $v'$ of $Y$. Then $H$ fixes pointwise every
combinatorial geodesic between $v$ and $v_0$.
\end{lemma}

\begin{proof}
We prove the result by induction on the combinatorial distance between $v$ and $v'$. If $d(v,v')=1$, the result is immediate, as there is unique edge between~$v$ and~$v'$. Suppose by induction that the result is true for vertices at distance at most $n \geq 1$, and let $v, v'$ be two vertices of~$Y$ at distance~$n+1$. Since $Y$ is systolic, it follows from \citep[Corollary~7.5]{JS} that the combinatorial ball of radius~$n$ around~$v'$, denoted $B(v', n)$, is a convex subset of~$Y$ in the sense of \citep[Definition~7.1]{JS}. Moreover, by \citep[Lemma 7.7]{JS}, this combinatorial ball intersects the combinatorial ball $B(v,1)$ along a single simplex. This implies that there exists a simplex~$\Delta$ of~$Y$ containing~$v$, and such that every combinatorial geodesic from~$v$ to $v'$ starts with an edge of~$\Delta$. In particular, we define~$\Delta'$ as the simplex of~$Y$ spanned by the first edges of \textit{all} the combinatorial geodesics from~$v$ to~$v'$. Since~$H$ fixes~$v$ and~$v'$, $H$~preserves the set of combinatorial geodesics from~$v$ to~$v'$, and in particular $H$ stabilises~$\Delta'$. Since $G$ acts on $Y$ without inversion, it follows that~$H$ fixes~$\Delta'$ pointwise. 

Let~$\gamma$ be a combinatorial geodesic from~$v$ to~$v'$. By the above, $H$~fixes the first edge~$e$ of~$\gamma$. Let~$v_1$ be the vertex of~$e$ distinct from~$v$. We have that~$H$ fixes~$v_1$ and~$v'$, and these two vertices are at combinatorial distance $n$. By the induction hypothesis, $H$ fixes pointwise the portion of~$\gamma$ between~$v_1$ and~$v'$, and it now follows that~$H$ fixes pointwise all of~$\gamma$. This concludes the induction.
\end{proof}

\noindent We proceed now to the proof of the main theorem of this section:

\noindent
\begin{proof}[Proof of \autoref{theorem:intersection}] 
We will prove the theorem by induction on the number $n$ of generators of~$A_S$.
If $n=2$, $A_S$ is an Artin group on two generators $a, b$ and there are two cases to consider. If $m_{ab}<\infty$, then $A_S$ is a spherical  Artin group, so \autoref{item1} follows from \citep[Theorem~9.5]{CGGW} and  \autoref{item2} follows from \citep[Theorem~0.2]{Godelle}. If $m_{ab}=\infty$, then $A_S$ is a free group on two generators $a, b$. Moreover, the proper parabolic subgroups are either trivial or infinite cyclic. Since the action of $A_S$ on the Bass-Serre tree associated to the splitting $\langle a \rangle * \langle b \rangle$ has trivial edge stabilisers, it follows that two distinct proper parabolic subgroups intersect trivially. Thus, \autoref{item1} and  \autoref{item2} follow immediately.

Let us now assume that the result is known for Artin groups of large type on at most $n$ generators with $n\geq 2$,
and let $A_S$ be an Artin group of large type on $n + 1$ generators. Let $X_S$ be its associated Artin complex.

\bigskip\noindent
\emph{Claim 1.}  Let $e_1,\dots,  e_k$ be a combinatorial path $p$ in $X_S$. Then there exists a simplex $\Delta$ of $X_S$ containing the edge $e_k$ such that

$$\bigcap_{1\leq i \leq k} \Stab_{X_S}(e_i) = \Stab_{X_S} (\Delta).$$

\noindent
\emph{Proof of Claim 1.} We will prove the claim by induction on $k$. If $k=1$, $p$ is just the edge $e_1$ and the proof is trivial. Now suppose that the claim is true for $k$ and let us prove it for $k+1$. By applying the induction hypothesis to the subpath $e_1, \ldots, e_k$, we will then have

$$\bigcap_{1\leq i \leq k+1} \Stab_{X_S}(e_i) = \Stab_{X_S} (\Delta') \cap \Stab_{X_S}(e_{k+1}),$$ where $\Delta'$ is a simplex containing the edge $e_k$. Let~$v$ be a vertex contained in both~$e_k$ and~$e_{k+1}$. By, \autoref{lemma:geometric_parabolics}, this means that both $\Stab_{X_S}(\Delta')$ and $\Stab_{X_S}(e_{k+1})$ are parabolic subgroups of $\Stab_{X_S}(v)$.  Also, up to conjugacy, $\Stab(v)$ is an Artin group on $n$~generators. Therefore, by the induction hypothesis on~$n$, $\Stab_{X_S} (\Delta') \cap \Stab_{X_S}(e_{k+1})$ is a parabolic subgroup of $\Stab(v)$ contained in $\Stab_{X_S}(e_{k+1})$, so it is a parabolic subgroup of $\Stab_{X_S}(e_{k+1})$. Geometrically, $\Stab_{X_S} (\Delta') \cap \Stab_{X_S}(e_{k+1})$ is the stabiliser of some simplex containing~$e_{k+1}$. This finishes the proof of Claim~1.

\bigskip\noindent
\emph{Claim 2.}  Let $\Delta_1$ and $\Delta_2$ be two simplices of $X_S$. Then there exists a simplex~$\Delta$ of~$X_S$ containing~$\Delta_2$ such that
$\Stab_{X_S}(\Delta_1) \cap \Stab_{X_S}(\Delta_2) = \Stab_{X_S} (\Delta).$

\medskip
\noindent
\emph{Proof of Claim 2.} 
Let $\Delta'$ be any simplex of $X_S$ and let $V_{\Delta'}$ be the set of vertices of $\Delta'$. As the action of $A_S$ on $X_S$ is without inversions, we have that $\Stab_{X_S}(\Delta')= \cap_{w\in V_{\Delta'}} \Stab(w)$. Define a combinatorial path $p$ that is the concatenation of the three following paths: a combiantorial path $p_1$ that travels along every vertex in $ V_{\Delta_1}$; a combinatorial geodesic $p_2$ between the endpoint of $p_1$ and $ V_{\Delta_2}$; and a combinatorial path that starts in the endpoint of $p_2$ and travels along every vertex in $ V_{\Delta_2}$. Denote the endpoint of $p$ by $v$ and let $E_p$ be the set of edges of~$p$. Then, by Claim~1 and \autoref{lemma:fixing_geodesic}, 

$$\Stab_{X_S}(\Delta_1) \cap \Stab_{X_S}(\Delta_2)= \bigcap_{w\in {V_{\Delta_1} \cup V_{\Delta_2}} } \Stab_{X_S}(w)= \bigcap_{e\in E_p} \Stab_{X_S}(e)= \Stab_{X_S}(\Delta),$$ for some simplex $\Delta$ containing $v$. Now we need to show that $\Delta$ contains also $\Delta_2$. Notice that $\Stab_{X_S}(\Delta_2)$ contains $\Stab_{X_S}(\Delta)$ and both $\Stab_{X_S}(\Delta_2)$ and $\Stab_{X_S}(\Delta)$ are parabolic subgroups of $\Stab_{X_S}(v)$. This group is, up to conjugacy, an Artin group on $n$ generators. So by using the induction hypothesis on $n$, $\Stab_{X_S}(\Delta)$ is a parabolic subgroup of $\Stab_{X_S}(\Delta_2)$, which means that we can choose~$\Delta$ to contain~$\Delta_2$. This finishes the proof of Claim~2.

\bigskip
In particular, note that Claim 2 together with \autoref{lemma:geometric_parabolics} implies that the parabolic subgroups of $A_S$ are stable under intersection, proving \autoref{item1}.

\bigskip

\noindent \emph{Claim 3.}  For every pair of simplices $\Delta_1$ and $\Delta_2$ of $X_S$ such that $\Stab_{X_S}(\Delta_1)\subseteq \Stab_{X_S}(\Delta_2)$, there exists a simplex $\Delta$ of $X_S$ containing $\Delta_2$ such that
$\Stab_{X_S}(\Delta_1) = \Stab_{X_S} (\Delta).$

\medskip
\noindent
\emph{Proof of Claim 3.} Just notice that $\Stab_{X_S}(\Delta_1)= \Stab_{X_S}(\Delta_1) \cap \Stab_{X_S}(\Delta_2)$, so by Claim 2 there is a simplex $\Delta$ containing $\Delta_2$ such that $\Stab_{X_S}(\Delta_1)=\Stab_{X_S}(\Delta)$. This finishes the proof of the claim.

\bigskip 

We now explain why this claim implies that $A_S$ satisfies \autoref{item2}. Let $P_1$ and $P_2$ be two parabolic subgroups of~$A_S$ such that $P_1 \subseteq P_2$. By \autoref{lemma:geometric_parabolics} there are simplices $\Delta_1$ and $\Delta_2$ of~$A_S$ such that $P_1=\Stab_{X_S}(\Delta_1)$ and $P_2=\Stab_{X_S}(\Delta_2)$. By Claim 3, there exists a simplex~$\Delta$ of~$X_S$ containing~$\Delta_2$ such that
$\Stab_{X_S}(\Delta_1) = \Stab_{X_S} (\Delta).$ Again by \autoref{lemma:geometric_parabolics}, this means that~$P_1$ is a parabolic subgroup of~$P_2$, as we wanted to prove.
\end{proof}

\begin{remark}
Notice that the only place where the systolic geometry was used in the previous proof is the following argument coming from \autoref{lemma:fixing_geodesic}: If an element fix two simplices, then it fixes pointwise a combinatorial path between these simplice. Therefore, a strong enough requirement to prove \autoref{theorem:intersection} for any Artin group~$A_S$ is to have this fixing-path condition in its Artin complex~$X_S$. 
\end{remark}

\begin{question*}
Let $X_S$ be the Artin complex of any Artin group $A_S$ and let $g\in A_S$ be an element fixing $\Delta_1$ and $\Delta_2$.  Is there a combinatorial path between $\Delta_1$ and $\Delta_2$ fixed by $g$ pointwise?
\end{question*}

\bigskip
We can generalise some interesting results concerning parabolic results that were previously shown for spherical Artin groups \citep[Section~10]{CGGW}: 
\begin{corollary}\label{cor:parabolic_closure}
	Let $A_S$ be an Artin group of large type.  Then, an arbitrary intersection of parabolic subgroup of $A_S$ is a parabolic subgroup. In particular,
	
	\begin{enumerate} 
	
		\item  For a subset $B\subset A_S$, there is a unique minimal parabolic subgroup of $A_S$ (with respect to the inclusion) containing $B$ ;
		
		\item The set of parabolic subgroups of $A_S$ is lattice with respect to the inclusion.
	\end{enumerate}

\end{corollary}

The strategy will be the same standard argument used in \citep[Proposition~10.1]{CGGW}. We can find the generalised FC version of the first statement for spherical parabolic subgroups in \citep[Corollary~3.2]{Rose}.

\begin{proof}
	 Let $\mathcal{P}$ be an arbitrary set of parabolic subgroups of $A_S$ and let $Q= \cap_{P\in\mathcal{P}} P$. $Q$ is contained in every parabolic subgroup in $\mathcal{P}$, so by \autoref{theorem:intersection}, we just need to prove that $Q$ is equivalent to a finite intersection of parabolic subgroups. Notice that every parabolic subgroup is expressed as the conjugate of some standard parabolic subgroup. Since~$A_S$ is a countable group and standard parabolic subgroups of $A_S$ are finite, the set of parabolic subgroups of $A_S$ is countable. In particular, $\mathcal{P}$ is countable. Enumerate the elements in $\mathcal{P}=\{P_1, P_2, P_3,\dots\}$ and let $$Q_m = \bigcap_{1\leq i\leq m} P_i.$$
	 By \autoref{theorem:intersection}, all $Q_m$'s belong to $\mathcal{P}$. As $Q=\cap_{i\in \mathbb{N}} Q_m$, we need to show that the set  $\{Q_m\, | \, m\in  \mathbb{N} \}$ is finite. 
	 
	 \medskip \noindent Let $X_S$ be the Artin complex of $A_S$.
	 Notice that we have a descending chain $$Q_1\supseteq Q_2\supseteq Q_3 \supseteq \dots $$ By doing an induction on the Claim 3 in the proof of \autoref{theorem:intersection}, one can easily see that  if $\Stab_{X_S}(\Delta_1) \supsetneq \Stab_{X_S}(\Delta_2) \supsetneq  \Stab_{X_S}(\Delta_3) \dots $, the dimension of~$\Delta_i$ has to be bigger than the dimension of~$\Delta_{i-1}$. As the dimension of $X_S$ is finite, the chain cannot be infinite. Therefore, $Q$~is the minimal parabolic subgroup on $\mathcal{P}$.
	 
	 \medskip\noindent To see the first statement, just assume that $\mathcal{P}=\{P \,|\,B \subset P\}$. For the second statement let~$P_1$ and $P_2$ be any two parabolic subgroups of $A_S$. We need a maximal parabolic subgroup $R_1$ contained in $P_1$ and $P_2$ and a minimal parabolic subgroup $R_2$ containing $P_1$ and $P_2$. By all the previous discussion, $R_1=P_1 \cap P_2$ and $R_2$ is the minimal parabolic subgroup in $\mathcal{P}$ when $\mathcal{P}=\{P\,|\,P_1 \cup P_2 \subseteq P\}$.	 
\end{proof}

%
%
%
%

\section{Normalisers and fixed-point sets of parabolic subgroups}\label{section5}

The goal of this section is to prove  \myref{thm:main_normalizer}{Theorem}.
%
In all this section we consider an Artin group~$A_S$ with $|S|\geq 3$. For a parabolic subgroup~$P$ of~$A_S$, we denote by $\Fix(P)$ (or  $\Fix_{X_S}(P)$ if we wish to highlight the ambient complex) the fixed-point set of $P$ in $X_S$. Since $A_S$ acts on $X_S$ without inversions, $\Fix(P)$ is a subcomplex of $X_S$. The connection between the normaliser $N(P)$ of a parabolic subgroup $P$ and its fixed-point set $\Fix(P)$  is given by the following:

\begin{lemma} \label{LemmaNorm} Let $P$ be a parabolic subgroup of $A_S$. Then the normaliser $N(P) $ of $P$  satisfies $$N(P) = \Stab(\Fix(P)).$$
	In addition, an element of $A_S$ belongs to $N(P)$ if and only if it sends some maximal simplex of $\Fix(P)$ to some maximal simplex of $\Fix(P)$.
\end{lemma}

\begin{proof} $(\subseteq)$ Let $g \in N(P)$, that is, $g P = P g$, and let $v \in \Fix(P)$. Then
	$$P \cdot (g \cdot v) = g \cdot (P \cdot v) = g \cdot v.$$
	In particular, $g \cdot v \in \Fix(P)$ and thus $g \in \Stab(\Fix(P))$.
	
	$(\supseteq)$ Let $g \in \Stab(\Fix(P))$ and let $\Delta \subseteq \Fix(P)$ be a maximal simplex in the sense that $\Stab(\Delta) = P$. Then $g \cdot \Delta \in \Fix(P)$, thus
	$$P \cdot (g \cdot \Delta) = g \cdot \Delta.$$
	In particular, $g P g^{-1}$ fixes $\Delta$, hence $g P g^{-1} \subseteq P$. In other words, $g \in N(P)$.
\end{proof}


The key geometric result to prove \myref{thm:main_normalizer}{Theorem} by means of studying fixed-point sets is the following: 

\begin{proposition} \label{PropFix}
		Let $A_S$ be a large-type Artin groups, and let $P$ be a  parabolic subgroup of $A_S$ of type $S'$.
	\begin{itemize}
		\item If $|S'|\geq 2$, then $\Fix(P)$ is a single simplex.
		\item If $|S'|=1$, then $\Fix(P)$ is a subcomplex whose dual graph is a simplicial tree (see \myref{def:dual}{Definition} for the terminology).
	\end{itemize}
\end{proposition}

\noindent

The proof of this proposition will be split into two cases. We first mention a useful observation that will allow for proofs by induction: 

\begin{lemma}\label{lem:link_fix}
	For a simplex $\Delta$ of $\Fix(P)$ 
	of type $S''$, the link $Lk_{\Fix(P)}(\Delta)$ is 
	isomorphic to~$\Fix_{X_{S''}}(P)$.
\end{lemma}

\begin{proof}
	We have $Lk_{\Fix(P)}(\sigma) = \Fix(P)\cap Lk_{X_S}(\sigma)$. Since $Lk_{X_S}(\sigma)$ is equivariantly isomorphic to~$X_{S''}$ by \autoref{LinkDev}, the previous intersection is thus isomorphic to $\Fix_{X_{S''}}(P)$.
\end{proof}

\paragraph{Parabolic subgroups on at least two generators.} We start with the case of a parabolic subgroup $P$ of type $S'$ with $|S'|\geq 2$. 

\begin{lemma} \label{LemmaSingleSimp} If $|S'| \geq 2$ then $\Fix(A_{S'})$ is a single simplex $\Delta$ such that $\Stab(\Delta)=A_{S'}$.
\end{lemma}

\begin{proof} We will be using the following claim:
	
\medskip	
		
		\noindent\emph{Claim.}\, If a subcomplex $Y$  of $X_S$ is such that all of its links are simplices or empty, then $Y$ itself is a simplex.

\smallskip		
		
		 Indeed, if $Y$ is not a simplex, then it contains a combinatorial path $u, v, w$ that forms a geodesic of $X_S$. The two vertices $u, w$ define two vertices of $Lk_{Y}(v)$ at distance at least $2$ by assumption, hence $Lk_{Y}(v)$ is not a simplex, which proves the claim. 
		
\medskip		
		
	Recall from \autoref{lem:link_fix} that for a simplex $\Delta$ of $\Fix(P)$ corresponding to a simplex of type $S''$, the link $Lk_{\Fix(P)}(\Delta)$ is isomorphic to $\Fix_{X_{S''}}(P)$.  If $|S-S'|=1$, then $\Fix(P)$ must be a single vertex $v$: If it weren't, it would follow from the convexity of $\Fix(P)$ (\autoref{lemma:fixing_geodesic}) that $P$ fixes an edge of~$X_S$, which is impossible since in that case $P$ is a maximal proper parabolic subgroup of~$A_S$. $\Fix(A_{S'})$ being a single simplex now follows by induction on $|S-S'|\geq 1$ by applying the above Claim. The dimension of $\Fix(A_{S'})$ is $|S-S'|-1$, so by maximality its stabiliser has to be $A_{S'}$. 
\end{proof}

\begin{corollary}\label{cor:stab_notcyclic} If $P$ is a parabolic subgroup of $A_S$ of type $S'$ with  $|S'| \geq 2$, then $N(P) = P$.
\end{corollary}

\begin{proof} By \autoref{LemmaNorm} we know that $N(P) = \Stab(\Fix(P))$. Moreover, we know from \autoref{LemmaSingleSimp} that there is a simplex $\Delta$ in $X_S$ such that $\Fix(P) = \Delta$ and $\Stab(\Delta)=P$. In particular,
	$$N(P) = \Stab(\Fix(P)) = \Stab(\Delta) {=} P.$$
\end{proof}

%

\paragraph{Parabolic subgroups on one generator.} We now move to the case of a parabolic subgroup of type $S'$ with $|S'|=1$. We start with the following general remark:

\begin{lemma} \label{LemmaConvex} Let $P$ be a parabolic subgroup of $A_S$. Then $\Fix(P)$ is contractible.
\end{lemma}

The proof of this lemma will rely on the  following notion of convexity from \citep{JS}:

\begin{definition}
	A subcomplex $Y$ of a simplicial complex $X$ is \textbf{$3$-convex} if it is full and every combinatorial geodesic of length $2$ with endpoints in $Y$ is contained in $Y$. It is \textbf{locally $3$-convex} if for every simplex $\sigma$ of $Y$, the link  $Lk_Y(\sigma)$ is $3$-convex in $Lk_X(\sigma)$.
\end{definition}

\begin{proof}[Proof of \autoref{LemmaConvex}]
%
%
By  \autoref{lemma:fixing_geodesic}, $\Fix(P)$ contains every geodesic between two vertices of $\Fix(P)$. In particular, it is connected and $3$-convex, hence locally $3$-convex by \citep[Fact~3.3.1]{JS}. By \citep[Lemma~7.2]{JS}, $\Fix(P)$ is thus contractible.
\end{proof}

It turns out that such fixed-point sets have a very simple geometry. We introduce the following:

\begin{definition}\label{def:dual}
	The \textbf{dual graph} $T_P$ of $\Fix(P)$ is defined as follows: 
	\begin{itemize}
		\item Vertices of $T_P$ correspond to the  simplices of $\Fix(P)$ of type $S'\subsetneq S$ with $|S'|=1$ (called \textbf{type 1 vertices}) or  $|S'|=2$ (called \textbf{type 2 vertices}).
		\item We put an edge between a type 1 vertex $\Delta$ and a type 2 vertex $\Delta'$ whenever $\Delta'\subset \Delta$. 
		\item Finally, $T_P$ is the subgraph obtained by removing the type 2 vertices that have valence 1.
	\end{itemize}
	We think of $T_P$ as a subgraph of the first barycentric subdivision of $\Fix(P)$.
\end{definition}

We have the following:

\begin{lemma}\label{lem:dualgraph_tree}
	The dual graph $T_P$ is a simplicial tree.
\end{lemma}

In a nutshell, the proof of \autoref{lem:dualgraph_tree} goes as follows: We construct a sequence of subcomplexes
$$X_0 \supsetneq X_1 \supsetneq \cdots \supsetneq X_k,$$
where $X_0 $ is the first barycentric subdivision of $\Fix(P)$ and $X_k=T_P$, and such that for each $0 \leq i \leq k-1$, $X_{i+1}$ is a deformation retract of $X_i$. Since $X_0$ is contractible by \autoref{LemmaConvex}, it will then follow that the graph $T_P$ is also contractible,  hence  is a tree. 

We will need the following standard result from algebraic topology to construct deformation retractions:

\begin{lemma}\label{lem:contractible_collapse}
	Let $X$ be a simplicial complex, and let $v$ be a vertex of $ X$ whose link ${Lk}_X(v)$ is contractible. Then the subcomplex spanned by $X-v$ is a deformation retract of $X$. 
\end{lemma}

\begin{proof}
	Since the star $\mathrm{Star}_X(v)$ is isomorphic to a cone over ${Lk}_X(v)$, we first notice that~$X$ is obtained from $X-v$ by coning-off the contractible link ${Lk}_X(v)$. Recall that for a simplicial complex~$Y$ and a contractible subcomplex~$Z$, the quotient map $Y \rightarrow Y/Z$ obtained by collapsing~$Z$ to a point is a homotopy equivalence, see  \cite[Proposition 0.17]{Hatcher}. We thus have the following commutative diagram:
	\[\begin{tikzcd}
	{X-v} && {X} \\
	\\
	{(X-v)/Lk_X(v)} && {X/\mathrm{Star}_X(v),}
	\arrow[from=1-1, to=3-1]
	\arrow[from=1-3, to=3-3]
	\arrow[Rightarrow, from=3-1, to=3-3, shorten <=11pt, shorten >=11pt, no head]
	\arrow[from=1-1, to=1-3, shorten <=7pt, shorten >=7pt, hook]
	\end{tikzcd}\]
	where both vertical arrows are homotopy equivalences since ${Lk}_X(v)$ and its cone $\mathrm{Star}_X(v)$ are contractible. Thus, the inclusion $X-v \hookrightarrow X$ is a homotopy equivalence, and it follows from \cite[Corollary 0.20]{Hatcher} that the subcomplex spanned by $X-v$ is a deformation retract of $X$.  
\end{proof}

\begin{proof}[Proof of \autoref{lem:dualgraph_tree}]
	Consider the barycentric subdivision $\Fix(P)'$ of~$\Fix(P)$. A vertex~$v$ of~$\Fix(P)'$ corresponds to a simplex of $\Fix(P)$; We will call  the dimension of the corresponding simplex the \textit{height} of~$v$. For every $0 \leq k \leq |S|-2$, we define the subcomplex $X_k$ of $\Fix(P)'$ spanned by the vertices of height at least $k$. In particular, $X_0 = \Fix(P)'$ and $X_{|S|-2}$ is a subgraph of $\Fix(P)'$ containing $T_P$. 
	
\smallskip	
	
	We now show that for every $0\leq k \leq |S|-3$, $X_{k+1}$ is a deformation retract of~$X_k$. Notice that~$X_k$ is obtained from~$X_{k+1}$ by adding for every vertex~$v$ of height~$k$ the star $ \mathrm{Star}_{X_{k}}(v)$, which is isomorphic to a simplicial cone over the link~${Lk}_{X_k}(v)$. Let~$v$ be a vertex of height $0 \leq k \leq |S|-3$.  This vertex corresponds to a simplex~$\Delta$ of~$\Fix(P)$ of type~$S'$ for some subset $S'\subsetneq S$ with $|S'|\geq 3$.  Note that a vertex of~$X_k$ adjacent to~$v$ must have height greater than~$k$ by construction, hence the link ${Lk}_{X_{k}}(v)$ is isomorphic to the first barycentric subdivision of ${Lk}_{\Fix(P)}(\Delta)$. In particular, ${Lk}_{X_{k}}(v)$  is isomorphic to the first barycentric subdivision of $\Fix_{X_{S'}}(P)$ by \autoref{lem:link_fix}, and hence is contractible by \autoref{LemmaConvex}. It thus follows from \autoref{lem:contractible_collapse} that~$X_{k+1}$ is a deformation retract of $X_{k+1} \cup \mathrm{Star}_{X_{k}}(v)$. Since for two distinct vertices $v$, $v'$ of height~$k$, the subcomplexes $X_{k+1} \cup \mathrm{Star}_{X_{k}}(v)$ and $X_{k+1} \cup \mathrm{Star}_{X_{k}}(v')$ intersect along $X_{k+1}$, we can glue the various deformation retractions into a deformation retraction of 
	$$X_{k} = X_{k+1} \cup \bigcup_{\mathrm{height}(v)=k}  \mathrm{Star}_{X_{k}}(v)$$
	onto $X_{k+1}$. Thus, for every $0\leq k \leq |S|-3$, $X_{k+1}$ is a deformation retract of $X_k$. Thus, the graph $X_{|S|-2}$ is a deformation retract of $X_0 = \Fix(P)'$. Since the latter complex is contractible by \autoref{LemmaConvex}, so is the graph $X_{|S|-2}$, and it follows that $X_{|S|-2}$ is a tree. Finally, $T_P$ is obtained from $X_{|S|-2}$ by removing the type 2 vertices that have valence 1. Thus, $T_P$ is a deformation retract of $X_{|S|-2}$, hence $T_P$ is a tree.
\end{proof}

Note that since $N(P) = \Stab(\Fix(P))$ by \autoref{LemmaNorm}, $N(P)$ acts on $\Fix(P)$, hence on the dual tree $T_P$. We will use this action to prove the following:

\begin{lemma}\label{lem:norm_split}
	The normaliser~$N(P)$ of~$P$ splits as a direct product $P\times F$, where~$F$ is a finitely generated free group. 
\end{lemma}

\begin{remark}
	It can be shown that the tree $T_P$ is $N(P)$-equivariantly isomorphic to the standard tree associated to~$P$ as considered in \citep[Definition~4.1]{MP1}. In particular, the proof of \autoref{lem:norm_split}  is essentially the same as the proof of \cite[Lemma~4.5]{MP1}. We however include a proof formulated in our setting for the sake of self-containment.
\end{remark}

Since $P$ is a normal subgroup of $N(P)$ acting trivially on~$T_P$ by construction of~$\Fix(P)$, we can look at the induced action of $N(P)/P$ on~$T_P$. We will use this action to completely describe the normaliser~$N(P)$. We first  need the following result: 

\begin{lemma}\label{lem:stab_tree}
For the action of $N(P)/P$ on $T_P$ we have: 
\begin{itemize}
	\item Type 1 vertices of $~T_P$ have a trivial stabiliser.
	\item Type 2 vertices of $~T_P$ have an infinite cyclic stabiliser.
\end{itemize} 
\end{lemma}

Before starting this proof, let us recall a standard result about dihedral Artin groups: 

\begin{lemma}[\citealp{BS}]\label{thm:centre_dihedral}
	Let $A_{ab}$ be a dihedral Artin group with $2<m_{ab}<\infty$, and let $\delta_{ab}$ be its \textbf{Garside element}, defined as follows: 
	$$\delta_{ab} = \underbrace{abab\cdots}_{m_{ab}}.$$
	Then the centre of $A_{ab}$ is infinite cyclic and equal to  $\langle \delta_{ab}\rangle$ if $m_{ab}$ is even, and $\langle \delta_{ab}^2\rangle$ otherwise. 
	\end{lemma}

\begin{proof}[Proof of \autoref{lem:stab_tree}]
A type 1 vertex $v$ of $T_P$ corresponds to a maximal simplex of $\Fix(P)$. Such a simplex has stabiliser $P$ by construction, hence $\Stab_{N(P)/P}(v)$ is trivial.
	
\noindent Let $v$ be a type 2 vertex of $T_P$ of type $\{c, d\}$. This vertex corresponds to a simplex with associated coset $gA_{cd}$ for some $g \in A_\Gamma$.  It follows from \citep[Lemma 4.5]{MP1} and  \autoref{thm:centre_dihedral} that we have: 
	\begin{itemize}
		\item If $m_{cd}$ is even, then $$\Stab_{N(P)/P}(v) = gZ(A_{cd})g^{-1} = \langle g\delta_{cd}g^{-1} \rangle; $$ 
		\item If $m_{cd}$ is odd, then $$\Stab_{N(P)/P}(v) = gZ(A_{cd})g^{-1} = \langle g\delta_{cd}^2g^{-1} \rangle,$$ 
	\end{itemize}
\end{proof}

We are now ready to prove \autoref{lem:norm_split}.

\begin{proof}[Proof of \autoref{lem:norm_split}]
	Since two type 1 vertices of $T_P$ corresponding to cosets of the same standard parabolic subgroup are in the same $N(P)$-orbit, hence in the same $N(P)/P$ orbit, it follows that the action of $N(P)/P$ on $T_P$ is cocompact. 
	
	Thus, $N(P)$ acts cocompactly and without inversion on a simplicial tree. By  \autoref{lem:stab_tree} the stabilisers of type 1 vertices are trivial (hence so are the stabilisers of edges) and the stabilisers of type 2 vertices are infinite cyclic. It thus follows from Bass-Serre theory that $N(P)/P$ is a finitely-generated free group, and thus~$N(P)$ splits as a direct product $P\times F$, where~$F$ is a finitely generated free group. 
\end{proof}

\paragraph{An explicit basis of the normaliser.} Finding an explicit basis for the free subgroup appearing in  \myref{thm:main_normalizer}{Theorem} is now a standard application of Bass-Serre theory, which was stated as a remark without further justification in \cite[Remark~4.6]{MP1}. We first start by describing a fundamental domain for the action, as well as the quotient space $T_P/N(P)$.

\begin{definition}
	Let $\Gamma'$ be the first barycentric subdivision of the Coxeter graph~$\Gamma_S$. A vertex of~$\Gamma'$ corresponding to a generator~$a$ of $A_S$ will be denoted~$v_a$ and will be said to be of  \textbf{type 1}, while a vertex of $\Gamma'$ corresponding to an edge of $\Gamma$ between generators~$a$ and~$b$ will be denoted $v_{ab}$ and will be said to be of \textbf{type 2}.
	
	 	Let $\Gamma_{a, \mathrm{odd}}$ denote the maximal connected subgraph of $\Gamma$ that contains the vertex~$a$ and only odd-labelled edges. 	Let~$\Gamma_P$ be the graph obtained from the disjoint union of all the edges of~$\Gamma'$ that contain a vertex of $\Gamma_{a, \mathrm{odd}}$, by the following identification: If such an edge~$e$ ($e'$ respectively) of $\Gamma'$ contains a vertex $v$ ($v'$ respectively) such that $v, v'$ correspond to the same vertex of $\Gamma_{a, \mathrm{odd}}$, then~$v$ and~$v'$ are identified and define the same vertex of~$\Gamma_P$.
\end{definition}

Some examples of the graph $\Gamma_P$ are given in  \autoref{fig:free_basis}, when the underlying Coxeter graph is a triangle.
\begin{center}
\begin{figure}
\begin{tabular}{|c|c|c|}
\hline
\multirow{2}{*}{Coxeter graph~$\Gamma$}                       &{Induced graph} &  Rank and basis of $F$, \\
       &  of groups on $\Gamma_p$  & with $N(P)\simeq P \times F$ \\ \hline
{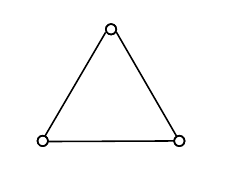} & {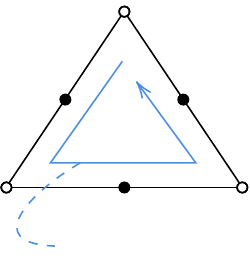} &  \begin{tabular}{c}

The rank is $4$ and a basis is \\

\small
 $\left\{\delta_{ab}^2, \delta_{ac}^2, \delta_{ac} \delta_{bc}^2 \delta_{ac}^{-1},\delta_{ab} \delta_{bc} \delta_{ac} \right\}$

\end{tabular} 
 \\ \hline
         {\input{Coxeter334.pdf_tex}}             &  {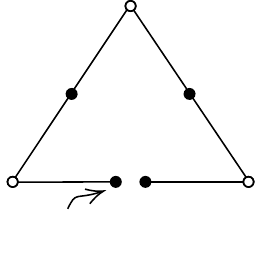} &  \begin{tabular}{c}

The rank is $4$ and a basis is \\         
         \small
$\left\{\delta_{ab}^2, \delta_{ac}^2, \delta_{ab} \delta_{bc} \delta_{ab}^{-1},\delta_{ac} \delta_{bc} \delta_{ac}^{-1} \right\}$          
\end{tabular}          
 \\ \hline
    {\input{Coxeter444.pdf_tex}} & {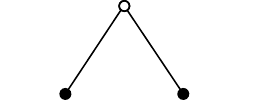} & \begin{tabular}{c} The rank is $2$  and a basis is \\ $\left\{\delta_{ab},\delta_{ac} \right\}$ 
\end{tabular}

 \\ \hline
    {\input{Coxeter443.pdf_tex}}  & {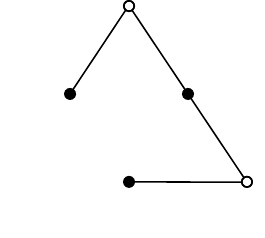} & \begin{tabular}{c}
  The rank is $3$ and a basis is \\ $\left\{\delta_{ab},\delta_{ac}^2, \delta_{ac} \delta_{bc} \delta_{ac}^{-1} \right\}$
\end{tabular}   
            \\ \hline
\end{tabular}

\caption{ Examples of computations of normalisers of the parabolic subgroup $P=\langle a \rangle$, for various large-type triangular Artin groups. Type 2 vertices of $\Gamma_P$ are indicated in bold in the second column and come with their infinite cyclic stabilisers. The group element in blue corresponds to the element of a basis of $F$ coming from the fundamental group of $\Gamma_P$. 
			Note that the structure of the normaliser for large-type triangular Artin groups depends only on the parity of the labels and not on the labels themselves, so the above cases cover all possible cases. }\label{fig:free_basis}
\end{figure}
\end{center}

\begin{definition}
	Let $e$ be an edge of $\Gamma_P$ between a type 1 vertex $v_c$ and a type $2$ vertex $v_{cd}$, for $c,d$ spanning an edge of $\Gamma$. We denote by $\widetilde{e}$ the edge of $T_P$ between the vertex~$A_c$ and the vertex $A_{cd}$. Choose an orientation of each edge of $\Gamma$. For each oriented loop  of $\Gamma_P$ based at $v_a$, we denote by $e_1, \ldots, e_n$ the oriented sequences of edges of $\Gamma$ crossed by $\gamma$, and we define 
$$g_\gamma \coloneqq \delta_{e_1}^{\pm1}\cdots \delta_{e_n}^{\pm1},$$
where the sign for each Garside element $\delta_{e_i}$ depends on whether~$\gamma$ follows the orientation of $e_i$.

We now choose a spanning tree $\tau$ of $\Gamma_P$, which we think of as being based at~$v_a$. For a vertex~$v$ of~$\Gamma_P$, we denote~$\gamma_v$ the oriented geodesic of~$\tau$ from~$v_a$ to $v$. Let~$e$ be an edge of $\Gamma_P$. If~$e$ is contained in~$\tau$, let $v$ be the vertex of $e$ closest to~$v_a$ in~$\tau$. If $e$ is not contained in~$\tau$, let~$v$ be the vertex of~$e$ closest to~$v_a$ in~$\Gamma_P$ (as $\Gamma_P$ is bipartite). We denote $g_v \coloneqq g_{\gamma_v}$, and we set 
$$Y_P \coloneqq \bigcup_{e \subset \Gamma_P} g_{v}\widetilde{e}.$$
This defines a connected subtree of $T_P$.
\end{definition}

\begin{lemma}
	The subtree $Y_P$ is a fundamental domain for the action of $N(P)$ on $T_P$, and the quotient $T_P/N(P)$ is isomorphic to $\Gamma_P$.
\end{lemma}

\begin{proof}
	An edge of $T_P$ corresponds to a pair consisting of a maximal simplex of $T_P$ (of type $c$ for some $c\in V(\Gamma)$) and one of its codimension 1 faces (of type $cd$ for some $d\in V(\Gamma)$ adjacent to~$c$). We thus mention the following useful fact, which is an immediate consequence of \autoref{LemmaNorm}:
	
	\medskip

	\textbf{Fact:} Two edges of $T_P$ in the same $A_S$-orbit are  also in the same $N(P)$-orbit. 
	
	\medskip
	
	\noindent  Let us first show that $Y_P$ is a fundamental domain for the action of $N(P)$ (and hence $N(P)/P$) on $T_P$. The fact that $Y_P$ is connected, hence a subtree of $T_P$, is a consequence of the construction. By construction of the various edges $\widetilde{e}$, it thus follows that the edges of $Y_P$ are in different $A_S$-orbits, and in particular in different $N(P)$-orbits. Now let $e$ be an edge of $T_P$. Its type 1 vertex is of type $c$, for some $c \in V(\Gamma)$ such that $\langle c \rangle$ and $\langle a \rangle$ are conjugated. It thus follows from \citep[]{Paris} that $c \in V(\Gamma_{a, odd})$, and it then follows that $e$ is in the $A_S$-orbit, hence the $N(P)$-orbit, of an edge of $Y_P$. Thus, $Y_P$ is a fundamental domain for the action of $N(P)$ (and hence $N(P)/P$) on~$T_P$.
	
    \bigskip 
    
     We now want to study the quotient space $T_P/N(P)$. Let us analyse the action of $N(P)/P$ on $T_P$ at a local level. 
	
	 \noindent Let $v$ be a vertex of $T_P$ of type $c \in V(\Gamma)$. By the above remark, we will assume up to to the action of $N(P)$ that this vertex corresponds to the  codimension 1 simplex of $X_S$ corresponding to $g_vA_c$. By construction of $T_P$, the codimension 1 faces of $\Delta$ that correspond to a type 2 vertex of $T_P$ adjacent to $v$ are the simplices corresponding to the parabolic subgroups $g_vA_{cd}$ with $d$ connected to $c$ in $\Gamma$. 
	
	\noindent  Let $v$ be a vertex of $T_P$ of type $\{c,d\}$ where $c, d$ span an edge of $\Gamma$. Up to the action of $N(P)$, we will assume that this vertex corresponds to the simplex with associated coset $g_vA_{cd}$.  Then it follows from  \autoref{lem:stab_tree} that we have: 
	\begin{itemize}
		\item If $m_{cd}$ is even, then  all the edges of $T_P$ containing $v$ are in the same $\langle \delta_{cd} \rangle$-orbit.
		\item If $m_{cd}$ is odd, then there are exactly two $N(P)$-orbits of edges of $T_P$ containing $v$,  corresponding to the $\langle \delta_{cd}^2\rangle$-orbits  of the maximal simplices of type  $\{c\}$ and $\{d\}$ respectively. 
	\end{itemize}
The description of the quotient $T_P/N(P)$ now follows from this local description.
\end{proof}

As mentioned earlier, the fundamental group $N(P)/P$ of this graph of groups over $\Gamma_P$ is a free group, and by Bass-Serre theory a basis for it is obtained by choosing a generator of each (infinite cyclic) stabiliser of vertex of dihedral type, as well as a family of elements corresponding to a basis of the fundamental group of $\Gamma_P$. We now explain how to construct explicitly these elements.


\begin{itemize}
	\item[1)] For each vertex $v$ of $Y_P$ of type $\{c, d\}$, a generator of $$\Stab_{N(P)/P}(v) = g_vZ(A_{cd})g_v^{-1}$$ is given by 
	$$\left\{
	\begin{array}{ll}
		g_{v}\cdot \delta_{cd}^{2}\cdot g_{v}^{-1} & \mbox{if  } m_{cd} \mbox{ is odd,} \\
		g_{v}\cdot \delta_{cd}\cdot g_{v}^{-1} & \mbox{otherwise.}
	\end{array}
	\right.$$
%
	\item[2)] A basis of $\pi_1(\Gamma_P)$ is in bijection with the edges of $\Gamma_P - \tau$. Let $e$ be such an edge, joining a type 1 vertex $v_c$ and a type 2 vertex $v_{cd}$, and let $e'$ be the edge joining $v_d$ and $v_{cd}$. Then 
the edges $g_{v_c}\delta_{cd}^{\pm 1}\widetilde{e}$ and $g_{v_{d}}\widetilde{e}'$ of $Y_P$ contain two type 2 vertices in the same $N(P)$-orbit, and the geodesic of $Y_P$ between these two vertices project to a loop of $\Gamma_P$ crossing $e$ exactly once that represents the element $$g_{v_{c}}\cdot \delta_{cd}^{\pm 1}\cdot g_{v_{d}}^{-1}\in N(P).$$
Note that this element is of the form $g_\gamma$, for some combinatorial $\gamma$ containing $e$. 
Thus, a family of elements for item 2) is given by the family of elements $g_\gamma$ when $\gamma$ runs over a basis of $\Gamma_P$.
\end{itemize}

We thus get the following:

\begin{corollary}\label{prop:basis_free}
	The normaliser $N(P)$ splits as a direct product  $N(P)=P\times F$, where $F$ is a finitely-generated free group with a basis given by the following family of elements: 
	\begin{itemize}
		\item for every vertex $v$ of $\Gamma_P$ of dihedral type $\{c,d\}$, the element $$\left\{
		\begin{array}{ll}
		g_{v}\cdot \delta_{cd}^{2}\cdot g_{v}^{-1} & \mbox{if  } m_{cd} \mbox{ is odd,} \\
		g_{v}\cdot \delta_{cd}\cdot g_{v}^{-1} & \mbox{otherwise.}
		\end{array}
		\right.$$
		\item for each combinatorial loop $\gamma$ based at $v_a$ in a chosen basis of $\Gamma_P$, the element $g_\gamma.$\qed
	\end{itemize}
\end{corollary}

In Figure 1, we give examples for various Artin groups associated to a triangular Coxeter graph of the normalisers of standard generators.

\section{Conjugacy stability and root stability}\label{section6}

We are now ready to prove \myref{conjugacystability}{Theorem} and \myref{roots}{Theorem}. In this section, $A_S$ denotes as usual an Artin group of large type on at least three generators.

By  \autoref{cor:parabolic_closure}, we can define the following subgroups of $A_S$:

\begin{definition}
Let $g \in A_S$. The minimal parabolic subgroup~$P_g$ containing~$g$ is called the \textbf{parabolic closure} of~$g$. 
\end{definition}

This subgroup behaves well under conjugacy as illustrated by the following result (which generalises an analogous statement for spherical Artin groups \citep[Lemma~8.1]{CGGW}):

\begin{lemma}\label{proposition:conjugacy_Pg}
 Let $g \in A_S$ and $\alpha \in A_S$. Then $$P_{\alpha^{-1} g \alpha} = \alpha^{-1} P_g \alpha.$$
	
	In particular, if $a$ and $b$ are conjugate, their parabolic closures correspond to stabilisers of simplices of $X_S$ with the same dimension.
\end{lemma}

\begin{proof}
	It is obvious that $\alpha^{-1} P_g \alpha$ contains $\alpha^{-1} g \alpha$. We need to prove that this parabolic subgroup is the minimal one containing $\alpha^{-1} g \alpha$.    Let $Q$ be any parabolic subgroup containing $\alpha^{-1} g \alpha$. As $\alpha Q \alpha^{-1}$ contains $g$, $P_g \subseteq \alpha Q \alpha^{-1}$.  Therefore, $ \alpha^{-1} P_g \alpha \subseteq Q$.
\end{proof}

\noindent We are finally able to prove the conjugacy stability theorem:

\begin{proof}[Proof of \myref{conjugacystability}{Theorem}]
	Let $g$ and $g'$ be two elements of $A_X$ that are conjugated by an element $\alpha\in A_S$.  As $P_g, P_{g'}\subset A_X$, by \autoref{theorem:intersection} there must be $Y,Y'\in X$ and $\beta,\beta'\in A_X$ such that $P_g=\beta^{-1}A_Y\beta$ and $P_g=\beta'^{-1}A_{Y'}\beta'$. Since~$P_g$ and~$P_g'$ are conjugate by  \autoref{proposition:conjugacy_Pg}, $A_Y$ and~$A_{Y'}$ have to be conjugate. At the beginning of this section, we have seen that if $|Y|>1$, then $Y=Y'$. Also, if $|Y|=1$, then either $Y=Y'$, or $Y$ and $Y'$ are single generators connected by an odd-labelled path in $\Gamma_S$.  Thus, there are two possibilities:
	
	\begin{itemize}
		
		\item Suppose that $P_g=\beta^{-1}A_Y\beta$ and $P_{g'}=\beta'^{-1}A_Y\beta'$, with $Y\subseteq X$ and $\beta,\beta'\in A_X$. Then $(\beta\alpha)^{-1}A_Y(\beta\alpha)=\beta'^{-1}A_Y\beta'$ and $\beta\alpha\beta'^{-1}$ normalises~$A_Y$. If the dimension of $A_Y$ is bigger than 1, then by \autoref{cor:stab_notcyclic}, $N(A_Y)=A_Y\subseteq A_X$, so $\alpha\in A_X$. If  the dimension of~$A_Y$ is~1, $g=\beta^{-1} a \beta, g'=\beta'^{-1} a\beta'$, for some $a\in X$, and they are conjugate by $\beta^{-1}\beta'\in A_X$.

		\item Suppose that $g=\gamma^{-1}a^n\gamma$ and $g'=\gamma'^{-1}b^n\gamma'$, $\gamma,\gamma'\in A_X $,  where $a$ and $b$ are Artin generators that are connected in $\Gamma_S$ by an odd-labelled path. Then, there is an element of $A_S$ conjugating~$a$ to~$b$. If there is an odd-labelled path in~$\Gamma_X$ connecting $a$ to $b$, then there is an element $c$ in $A_X$ that conjugates $a$ to $b$. Thus, $\gamma^{-1}c\gamma'$ conjugates $g$ to $g'$ . 
		
		On the contrary, if there is no such a path in $\Gamma_X$, there is no element in $A_X$ conjugating~$a$ to $b$. Since the parabolic closures of~$g$ and $g'$ are respectively $\gamma^{-1} \langle a \rangle\gamma$ and $\gamma'^{-1} \langle b \rangle\gamma'$, by \autoref{proposition:conjugacy_Pg}  there is no element in $A_X$ conjugating $g$ to $g'$. This is then the only case in which $A_X$ is not conjugacy stable in $A_S$.
	\end{itemize}
\end{proof}	
	
\noindent	
We also prove that the parabolic closure of an element $g$ is stable when taking roots and powers of $g$. This is a generalisation of \citep[Corollary~8.3]{CGGW}.
	
	\begin{proposition} \label{PropositionPowerStability} Let $A_S$ be a large-type Artin group of rank at least $2$, and let $g \in A_S$. Then for every $n \in \mathbf{Z} \backslash \{0\}$ we have $P_g = P_{g^n}$.
	\end{proposition}
	
Before coming to the proof of this proposition, we first introduce the following Lemma. Note that this result and its proof are analogous to \citep[Theorem 7.3]{chepoidismantlability}:

\begin{lemma} \label{LemmaExistsInvariantSimplex} Let $G$ be a group acting by simplicial automorphisms on a systolic complex $X$. Suppose that there is a vertex $v \in X$ whose orbit $G v$ is finite. Then there exists a simplex of $X$ that is invariant under the action of $G$.
\end{lemma}

\begin{proof} The statement of \citep[Theorem 7.3]{chepoidismantlability} is given for a finite group $G$. However, their proof only uses the finiteness of $G$ to obtain a finite $G$-orbit, out of which they construct an invariant simplex. In particular, their proof generalises without any change to the case of an infinite group $G$ with a finite $G$-orbit. 
\end{proof}

We now come to the proof of \autoref{PropositionPowerStability}:

	\begin{proof}[Proof of \autoref{PropositionPowerStability}]
	We show by induction on $|S|$ that $P_g = P_{g^n}$. If $|S| = 2$, $A_S$ is a dihedral Artin group. In particular, it is spherical, and the result follows from \citep[Corollary~8.3]{CGGW}. Let now $|S| \geq 3$, and suppose that $P_g \neq P_{g^n}$. We have that $P_{g^n} \subseteq P_g$, and then there is a chain of inclusions of the form
	$$P_{g^n} \subsetneq P_g \subseteq A_S.$$

\noindent \emph{Claim.}  We have $P_g \subsetneq A_S$.

\medskip

 Indeed, since $P_{g^n} \subsetneq A_S$, the set $\mathrm{Fix}_{X_S}(P_{g^n})$ is non-empty. In particular, $g^n$ is elliptic, and thus $g$ has finite orbits, as for every point $v \in \mathrm{Fix}(g^n)$,
	$$\langle g \rangle \cdot v = \{v, g v, g^2 v, \cdots, g^{n-1} v \}.$$
	By \autoref{LemmaExistsInvariantSimplex}, $g$ must stabilise some simplex $\Delta$ in $X_{S}$. Because the action of $A_S$ on $X_S$ is without inversions, $g$ must fix $\Delta$ pointwise. In other words, $\mathrm{Fix}(g)$ is non-empty, hence $P_g \subsetneq A_S$. This finishes the proof of our claim. 
	
	\medskip
	
Also, we have $P_g = h A_{S'} h^{-1}$ for some $h \in A_S$ and $S' \subsetneq S$. Now notice that
	$$h^{-1} P_{g^n} h \subsetneq h^{-1} P_{g} h = A_{S'},$$
	and thus $P_{h^{-1} g^n h} \subsetneq P_{h^{-1} g h} = A_{S'}$ by \autoref{proposition:conjugacy_Pg}. 
As $|S'| < |S|$, we can use the induction hypothesis on $X_{S'}$. This yields $P_{h^{-1} g h} = P_{h^{-1} g^n h}$. In particular, one has $P_g = P_{g^n}$ by \autoref{proposition:conjugacy_Pg}, which is a contradiction.
	\end{proof}
	


\noindent
As an immediate consequence, we have the following result:	
		
	\begin{corollary} Let $A_S$ be a large-type Artin group of rank at least $2$, and let $P$ be a  parabolic subgroup of $A_S$. If $g^n \in P$ for some $n \in \mathbf{Z} \backslash \{0\}$, then $g \in P$.	\qed
	\end{corollary}

\bigskip

\noindent{\textbf{\Large{Acknowledgments}}} 

\medskip

The authors were partially supported by the EPSRC New Investigator Award EP/S010963/1. The first author was partially supported by the research grants MTM2016-76453-C2-1-P (financed by the Spanish Ministry of Economy and FEDER) and US-1263032 (financed by the Andalusian Ministry of Economy and Knowledge, and the Operational Program FEDER 2014--2020).

\bibliography{Bib_Systole}

\bigskip\bigskip
\noindent
\textit{Mar\'{i}a Cumplido*,} \texttt{\href{mailto:maria.cumplido.cabello@gmail.com}{maria.cumplido.cabello@gmail.com}}.

\smallskip\noindent
\textit{Alexandre Martin*,} \texttt{\href{mailto:Alexandre.Martin@hw.ac.uk}{Alexandre.Martin@hw.ac.uk}}.

\smallskip \noindent
\textit{Nicolas Vaskou*,} \texttt{\href{mailto:ncv1@hw.ac.uk}{ncv1@hw.ac.uk }}.

\medskip \noindent
\textit{*Department of Mathematics, Heriot-Watt University, Riccarton, EH14 4AS Edinburgh, \\Scotland, UK.}

\end{document}